\documentclass[12pt,final,reqno]{amsart}
\usepackage{ntung}
\usepackage{tikz,verbatim,xcolor,graphicx}
\usepackage{amsthm}
\usepackage{thmtools}
\usepackage{thm-restate}

\usepackage{ntung_commands}

\DeclareMathOperator{\ark}{AR}

\DeclareMathOperator{\subspan}{span}

\newcommand{\F}{\mathbb{F}}

\begin{document}

\title{Randomly piercing algebraic sets}
\author[A1]{Daniel Altman}
\address{Department of Mathematics, Stanford University, CA 94305, USA}
\email{daniel.h.altman@gmail.com}
\author[A2]{Nathan Tung}
\address{Department of Statistics, Stanford University, CA 94305, USA}
\email{ntung@stanford.edu}
\begin{abstract}
We show, for example, that if one samples
\[\frac{\log p}{2\log(1+(p-1)^{-1})} \cdot n^2(1 + o_{n\to \infty}(1))\]
points in $\F_p^n$ at random then asymptotically almost surely this set intersects every quadratic hypersurface. We furthermore show that this is tight in that sampling $o_{n\to\infty}(n^2)$ fewer points almost surely fails to intersect some quadratic hypersurface.

Our main result is a sharp threshold  for the following problem: how many points in $\F_p^n$ does one need to randomly sample to almost surely intersect every algebraic set defined by at most $s$ polynomials each of degree at most $k$? As an application we improve lower bounds in the random Szemer\'{e}di theorem in $\F_p^n$, in particular obtaining a leading constant which grows as the threshold for what is considered a `dense' set in Szemer\'{e}di's theorem shrinks.
\end{abstract}
\maketitle

\section{Introduction}\label{sec:intro}

For an odd prime $p$ and collection of polynomials $f_1,\dots,f_s \in \FF_p[x_1,\dots,x_n]$ let 
$$V(f_1,\dots,f_s) = \set{x \in \FF_p^n: f_1(x) = \dots = f_s(x) = 0}
$$ 
be the variety\footnote{In this paper we use the term ``variety'' rather liberally to simply mean the set of solutions to a system of polynomial equations over a given (finite) field. In particular our use of this term does not imply irreducibility, nor does it imply that the space of solutions is to be taken in the algebraic closure.} determined by the vanishing of $f_1,\dots,f_s$. Let $\cF_{k,s}$ be the collection of all such varieties defined by $s$ polynomials $f_i$ of degree at most $k$. In this paper we determine in Theorem \ref{thm:mainthresh} how many random points in $\FF_p^n$ one needs to sample to intersect every element of this collection. 

We say a set $M$ pierces a collection of sets $\cF$ if $M \cap F \neq \emptyset$ for every $F \in \cF$. Thus, put another way, we determine a threshold for the size of a random set that pierces $\cF_{k,s}$ with high probability. Throughout this paper we deal with the regime where $p$ is fixed and $n$ is large. All asymptotic notation can be interpreted relative to the limit $n\to \infty$ unless subscripts denote otherwise. In particular, our threshold result is with respect to this limit. 

Before stating the main theorem in generality, we will briefly discuss the simplest nonlinear case ($k=2,s=1$). For further simplicity, we will restrict this discussion to homogeneous polynomials. Here,
\begin{equation}\label{eq:quadhyp}
    \widetilde{\cF}_{2,1} = \set{V(f): f(x) = x^\top A x, A \in \text{SymMat}_{n\times n}(\FF_p)}.
\end{equation}
How many random points in $\F_p^n$ are needed to almost surely pierce $\widetilde{\cF}_{2,1}$, the family of homogeneous quadratic hypersurfaces? Let's begin with what one \textit{expects}.

For fixed nonzero $x \in \FF_p^n$, the value $x^\top A x$ distributes uniformly over $\FF_p$ as $A$ ranges over the set $\text{SymMat}_{n\times n}(\FF_p)$. On average then ${\abs{V(f)}}/p^n = 1/p$,
and if $X_{\abs{M}}$ denotes the number of hypersurfaces in $\wt \cF_{2,1}$ disjoint from a random set $M \subseteq \FF_p^n$, then perhaps one expects
$$
\EE X_{\abs{M}} \approx \abs{\widetilde{\cF}_{2,1}} (1-p^{-1})^{\abs{M}} \approx p^{n+1 \choose 2} (1-p^{-1})^{\abs{M}},
$$
which yields a threshold for this first moment being of constant order around
$$
\abs{M} \approx \frac{\log p}{\log\bigp{1+(p-1)^{-1}}}{n \choose 2}.
$$
Indeed, this turns out to be the threshold for $\cF_{2,1}$ as seen by taking $k=2, s=1$ in our main theorem, which we now state.

\begin{theorem}\label{thm:mainthresh}
Let $M \subset \FF_p^n$ be a set chosen uniformly at random among sets of fixed size $\abs{M}$. Let $k\geq 2$ be a positive integer with $k<p$. Let $\cF_{k,s}$ be the collection of nonempty varieties determined by the vanishing of $s$ polynomials which each have degree at most $k$. Then letting
    $$
    c_{k,s} = \frac{p^{(s-1)k}\log p}{\log(1+(p-1)^{-1})},
    $$
    we have for every $\eps > 0$
    \[ \pr{M \text{ pierces } \cF_{k,s}} = \begin{cases} 
          1-o(1), & |M| \geq (c_{k,s}+\eps) {n \choose k}\\
          o(1), & |M| \leq (c_{k,s}-\eps) {n \choose k} 
       \end{cases},\]
    as $n\to \infty$. Note that $c_{k,s} = (1+o_{p\to \infty}(1))p^{(s-1)k+1}\log p$.
\end{theorem}

Only in the case of hypersurfaces does the  threshold coincide with the first moment threshold: for $s \geq 2$, $X_{\abs{M}}$ does not concentrate around its expectation. There is a regime above the probability threshold but below the first moment threshold where with high probability $X_{\abs{M}} = 0$ ($M$ pierces $\cF_{k,s}$) but $\EE X_{\abs{M}} \to \infty$. Certainly in this setting we need more than a vanilla second moment method. Indeed, Theorem \ref{thm:mainthresh} is proved via several more refined thresholds for piercing sub-collections that partition $\cF_{k,s}$.

\subsection{Application to Szemerédi's theorem with random differences}\label{sec:szres}

As an application of our arguments, we improve lower bounds on the size of a set of random differences for which Szemer\'{e}di's theorem on $k$-term arithmetic progressions ($k$-APs) holds with high probability.\footnote{We caution the reader that the numerology is such that the degree of the polynomials that are being pierced and the length of the corresponding arithmetic progression it is applied to are out by one: degree $2$ polynomials yield a result about $3$-APs, etcetera.} This problem has a rich history in the integers (see, e.g., \cites{ruzsa1978difference,ondiffsumint,bour,christ,rszsurvey,bdg19,bcs2024}) and more recently over finite fields. We will just discuss the latter, and refer the reader to (say) the introduction of \cite{bcs2024} for a summary of the state of affairs on the former. Here our ambient group is $\F_p^n$ (where again we think of $p$ as being fixed and $n$ as being large), and we denote the size of the group by $N:= p^n$.

We say \textit{Szemer\'{e}di's theorem holds for $k$-APs and $\alpha$-dense sets with differences in $M$} if there exists $N_0(\alpha,k)$ such that for every $N > N_0(\alpha, k)$, and subset $A \subseteq \F_p^n$ with $\abs{A} \ge \alpha N$, $A$ contains a nontrivial $k$-term AP with common difference in $M$.

Since Szemer\'{e}di's theorem holding with differences in $M$ is a monotone property (adding more elements to $M$ only makes Szemer\'{e}di's theorem easier to satisfy), there is a probability threshold function $\cP(\alpha,k)$, which we define to be the minimal density such that when $M$ is sampled uniformly with $\abs{M} = \cP(\alpha,k) N$, the probability of Szemer\'{e}di's theorem holding for $k$-APs and $\alpha$-dense sets with differences in $M$ is at least $1/2$.  

It was shown by the first author \cite{altfp} that for $\alpha = p^{-2}$, and sufficiently large $n$,
$$
\cP(\alpha,3)N \ge \frac{1}{2} (\log_p N)^2 - 11\log_p N\log_p\log N.
$$ In particular, this demonstrated that the finite field threshold is necessarily larger than what is conjectured in the integers $\{1,\ldots, N\}$, where it is believed \cite{rszsurvey} (with analogous definitions) that $\cP(\alpha,k)N = \Theta_{\alpha,k} (\log N)$ for all $k$ (this is only known for $k=2$, due to Bourgain \cite{bour}). Bri\"et \cite{briet21} then showed that for $\alpha=p^{-k}$, and sufficiently large $n$,
$$\cP(\alpha,k)N \ge \frac{1}{(k-1)!}( \log_p N)^{k-1} - C_{p,k}(\log_p N)^{k-2}(\log_p\log N)^2.$$ 
Recently, Zheng \cite{zheng25} obtained a small improvement by combining a generalisation of the argument of \cite{altfp} with improved bounds on the partition rank of a tensor in terms of its analytic rank \cite{moshkovitz2024quasilinearrelationpartitionanalytic}:
$$\cP(\alpha,k)N \ge \frac{1}{(k-1)!}( \log_p N)^{k-1} - C_{p,k}(\log_p N)^{k-2}(\log\log N)^{1+\eps},$$
where again one may take $\alpha = p^{-k}$. 

On the other hand, best-known upper bounds, which are due to Bri\"et--Castro-Silva \cite{bcs2024}, are of the following form.
\begin{theorem}[{\cite{bcs2024}}]
With notation as above, we have 
\[ \cP(\alpha, k)N \leq C_{\alpha, k} N^{1-2/k}\log^{2k+1} N.\]
\end{theorem}

The best-known upper bounds over the integers are of a similar shape. By analogy to the aforementioned conjecture over the integers (see \cite{rszsurvey}*{Ch. 3}, but also \cite{bl21} and \cite{bg22} for some evidence to the contrary), there is a belief that in $\F_p^n$ these lower bounds may in fact be the truth in terms of dependence on $N=p^n$. That is, one may not be able to hope for an improvement on the above lower bounds by more than a multiplicative constant. Indeed, in Proposition \ref{prop:densityupper} we note that any such improvement must necessarily consider obstructions additional to  the higher-order ones analysed in this paper. On the other hand, it is notable that the above lower bounds are independent of the density $\alpha$, which is a rather disappointing state of affairs for a Szemer\'edi-type theorem. Our application is to rectify this situation by improving the leading constant from $\Omega_{p,k}(1)$ to $\Omega_{p,k}(\alpha^{-1}\log \alpha^{-1})$. More precisely, we prove the following in Section \ref{sec:rszproof}.

\begin{restatable}[Random Szemer\'{e}di lower bound]{theorem}{rsz}\label{thm:rszbound}
     Let $p > k \ge 3$ be fixed. Let $\alpha \in (0,1/p)$, and let $t$ be the positive integer such that $\alpha \in [p^{-(t+1)}, p^{-t})$. Then
    $$
    \cP(\alpha,k) p^n \ge (1-o(1))\frac{\log (p^t)}{\log(1+(p^t-1)^{-1})}{n \choose k-1},
    $$
    or with $N = p^n$,
    \begin{align*}
        \cP(\alpha,k)N &\ge \bigp{1-o(1)}\frac{\log (p^t)}{(k-1)!\log(1+(p^t-1)^{-1})} (\log_p N)^{k-1} \\
        &\geq  \bigp{1-o_{\substack{\alpha \to 0 \\ n \to \infty}}(1)}\frac{\alpha^{-1}\log\alpha^{-1}}{p(k-1)!}(\log_p N)^{k-1}.
    \end{align*}
\end{restatable}

We note that the same methods prove an analogous bound for the case $k=2$, though many aspects of the argument simplify in this case so we omit the details. Instead we refer the reader to forthcoming followup work of the second author \cite{nt26} which will work out many of these details in the generality of finite abelian groups.

\subsection{Further applications and directions}
Although Theorem \ref{thm:mainthresh} is stated for $p$ fixed, its proof tolerates growing torsion. Of course, algebraic sets are then no longer dense in this setting, and if one wishes to prove a version of Theorem \ref{thm:rszbound} in such groups, one should instead consider the \textit{approximate} vanishing of polynomials, i.e., establish bounds for the piercing of (higher-order) Bohr sets. We don't pursue this in this document, but expect that the strategy here generalises. Indeed, the 2-AP case is undertaken in \cite{nt26}, where one replaces polynomial phases/higher-order characters with linear phases/characters.

 We furthermore note the relevance of the above line of enquiry to questions about the chromatic and independence numbers of random Cayley graphs. For example, Alon \cite{alon} proves a bound on the densities at which random Cayley graphs are $3$-colourable by proving a bound on  the threshold at which radius $\frac{1}{3}$ Bohr sets are pierced. As mentioned in the previous paragraph, our methods are amenable to understanding the piercing of Bohr sets, as is demonstrated in forthcoming followup work of the second author \cite{nt26} which improves this bound in the generality of finite abelian groups, thus improving corresponding bounds on the densities at which random Cayley graphs are $3$-colourable in this setting.

\subsection{Organisation of the paper.} In Section \ref{s:prelim} we compile some preliminary ingredients. In Section \ref{s:relative}, we prove via a second moment method Proposition \ref{prop:relativethresh}, a suitably quantitative threshold result for the piercing of high rank varieties relative to a fixed low rank variety. In Section \ref{s:main-proof}, we bootstrap Proposition \ref{prop:relativethresh} in a modified second moment method, yielding the proof the Theorem \ref{thm:mainthresh}. We conclude in Section \ref{sec:rszproof} with the proof of Theorem \ref{thm:rszbound}, the application to the random Szemer\'edi problem.

\subsection{Notation, conventions.} We use Vinogradov notation $\gg$, $\ll$ in addition to big-$O$ notation. We will also use the asymptotic notation $\omega$, which may be defined by the property that $f = o(g)$ if and only if $g = \omega(f)$. All asymptotic notation may be understood to refer to the limit $n\to \infty$ unless specified otherwise. We will say that an event $E=E_n$ occurs \textit{with high probability} if $\PP(E) = 1-o(1)$. We have parameters $p,k$ throughout the paper; it is always tacitly assumed that $p>k$.

\subsection*{Acknowledgements.} The authors thank Dor Elboim for helpful conversations and Sarah Peluse for generous feedback on an earlier version of this document.

\section{Preliminaries}\label{s:prelim}

Let $k\geq 2$. We will denote by $\F_p[x_1,\ldots, x_n]_{\leq k}$ the subspace of $\F_p[x_1,\ldots,x_n]$ comprising polynomials of degree at most $k$.

First we record some relations between polynomials and multilinear forms. For a polynomial $f$ in $n$ variables over $\F_p$ and $h \in \FF_p^n$, let the discrete derivative $\Delta_h$ be defined by
$$
(\Delta_h f)(x) := f(x+h)-f(x).
$$
Next we define the $k$-fold multilinearisation map
$\phi_k$, which takes as input $f\in  \F_p[x_1,\ldots,x_n]_{\leq k}$  and outputs a symmetric $k$-linear form $\phi_k(f):=T : \FF_p^n \times \dots \times \FF_p^n \to \FF_p$ defined by 
\begin{equation}\label{eq:multi}
    T(h_1,\dots,h_k) \coloneqq (\Delta_{h_1}\dots\Delta_{h_k}f)(x) = \sum_{\eps \in \set{0,1}^k} (-1)^{k-\abs{\eps}} f(\eps \cdot h),
\end{equation}
where $\eps \cdot h = \sum_i \eps_i h_i$ and $\abs{\eps}$ is the number of ones in $\eps$. See \cite[Lemma 2.4]{gowolf} for a proof that $T$ as defined in the previous sentence is indeed a symmetric $k$-linear form independent of $x$. We note that $\phi_k$ is linear. Then the \textit{bias} of $T$ is defined by
$$
\text{bias}(T) := \left|\EE_x e_p(T(x_1,\ldots, x_k))\right|,
$$
where probabilities and expectations are uniform over $x \in (\FF_p^n)^k$, and $e_p(\cdot)$ is shorthand for $e^{2\pi i \cdot / p}$. The analytic rank of $f$ and $T$ are then defined by
$$
\ark(f) \coloneqq \ark(T) \coloneqq -\log_p \text{bias}(T).
$$

For more on the background established thus far in this section, we encourage the reader to consult \cite{gowolf}. 

Note that if $f$ is a polynomial of degree strictly less than $k$, then $T(h_1,\dots,h_k) \equiv 0$ so $\ark(f) = 0$. Going forward it will be clear from context what $k$ we are using in this $k$-fold multilinearisation, or when we refer to the analytic rank of a given polynomial.

\begin{lemma}[Low rank count]\label{lem:lowrankcount}
    For $k < p < r$ there are at most $p^{O_k(n^{k-1}r \log_p r)}$ polynomials $\FF_p^n \to \FF_p$ with degree at most $k$ and analytic rank at most $r$.
\end{lemma}
\begin{proof}
    The number of $k$-multilinear forms $T$ for which $\ark(T) \le r$ is at most $p^{O_k(n^{k-1}r \log_p r)}$. This is shown, for example, in \cite[Section 2]{zheng25}. 
    
    Recalling that $\phi_k$ is linear, we see that if $f$ is homogeneous of degree $k$ then
    $$
    T(x,\dots,x) = \sum_{i=0}^k (-1)^{k-i} {k \choose i} f(ix) = f(x) \sum_{i=0}^k (-1)^{k-i} {k \choose i} i^k = k! \cdot f(x).
    $$
    Thus, the kernel of $\phi_k$ intersects the subspace of homogeneous degree $k$ polynomials trivially (we use here $k! \not\equiv 0 \pmod{p}$). Therefore the kernel of $\phi_k$ is the subspace of degree $\leq k-1$ polynomials, which has dimension $O_k(n^{k-1})$. This completes the proof. 
\end{proof}

We will need a lemma of Gowers and Wolf that controls exponential sums by the analytic rank of their phases.

\begin{lemma}[Exponential sum control \cite{gowolf}*{Lemma 3.2}]\label{lem:gauss} Let $f \in \F_p[x_1,\ldots, x_n]_{\leq k}$. Then
    $$
    \abs{\EE_{x \in \FF_p^n} e_p(f(x))} \le p^{-\ark(f)/2^{k-1}}.
    $$
\end{lemma}

Given polynomials $f_1,\ldots, f_s\in \F_p[x_1,\ldots, x_n]$, note that
$$
V(f_1,\dots,f_s) = V(\subspan_{\F_p}(f_1,\dots,f_s)).
$$
It is rather more convenient to deal with algebraic sets defined by the vanishing of subspaces of polynomials, and henceforth this will generally be the perspective we adopt.

Next we will need the Chevalley--Warning theorem. The statement below may be found, for example, in \cite[Theorem 6.11]{ln}.

\begin{lemma}[Chevalley--Warning theorem]\label{lem:warning} Let $S \leq \F_p[x_1,\ldots, x_n]_{\leq k}$ be a subspace of polynomials and let $s:= \dim(S)$. If $\abs{V(S)} \ge 1$ then
    $$
    \abs{V(S)} \ge p^{n-sk}.
    $$
\end{lemma}

We will also need that the above bound is tight. We include a proof for the convenience of the reader.

\begin{lemma}[Smallest nonempty variety]\label{lem:warn-tight}
    Let $s, k \geq 1$ and $n\geq sk$. Then there is a subspace $S\leq \F_p[x_1,\ldots, x_n]$ whose nonzero elements are each homogeneous of degree $k$, which has $\dim (S) = s$, and with $\abs{V(S)} = p^{n-sk}$.
\end{lemma}
\begin{proof}
    We claim that it suffices to identify a degree $k$ polynomial $f: \F_p^k \to \F_p$ which vanishes only at zero. Indeed, then we can let $S$ be the subspace spanned by polynomials $f_i$ for $i=1, \ldots, s$, where $f_i$ acts as $f$ on the variables $x_{(i-1)k+1}, \ldots, x_{ik}$ and is independent of the remaining $n-k$ variables. Then a point $x\in \F_p^n$ lies in $V(S)$ if and only if its first $sk$ coordinates are zero. 

    The polynomial $f$ is constructed as a norm form. Let $q=p^k$ and let $\alpha_1, \ldots, \alpha_k$ be a basis for $\F_q$ over $\F_p$. Then let 
    \[ f(x_1,\ldots, x_k) = \prod_{i=1}^k \sigma_i\big( x_1\alpha_1 + \cdots + x_k\alpha_k \big) = \prod_{i=1}^k \big( x_1\sigma_i(\alpha_1) + \cdots + x_k\sigma_i(\alpha_k) \big),\]
    where $\sigma_1$ is the $p$th power map (the Frobenius automorphism), and $\sigma_i=\sigma_1^i$ ($i$-fold composition, so $\sigma_i$ is the $p^i$th power map). Then, since $f$ is fixed by the Frobenius automorphism, we have that indeed $f(x_1,\ldots, x_k)\in \F_p$ and it is clear that $f$ has degree $k$. Furthermore, since $\{\alpha_i\}$ is a basis and each $\sigma_i$ is an automorphism of $\F_q$, we see that $f(x_1,\ldots, x_k)=0$ only if $x_1 = \cdots = x_k = 0$.
\end{proof}

Finally, we will need the triangle inequality for analytic rank \cite{lovetriangle}.

\begin{lemma}[Triangle inequality]\label{lem:triangle}
    For two polynomials $f,g \in \F_p[x_1,\ldots,x_n]_{\leq k}$ we have that
    $$
    \ark(f + g) \le \ark(f) + \ark(g)
    $$
\end{lemma}
\begin{proof}
    Recalling that $\phi_k$ is linear and that $\ark(f) := \ark(\phi_k(f))$, the result follows from the triangle inequality for the analytic rank of multilinear forms, due to Lovett \cite{lovetriangle}.
\end{proof}

\section{Relative results}\label{s:relative}
The threshold for piercing $\cF_{k,s}$ as in Theorem \ref{thm:mainthresh} is of course the same as for the collection of varieties generated by subspaces of polynomials of dimension at most $s$ containing  polynomials of degree at most $k$. We will work with this incarnation. 

The main result of this section Proposition \ref{prop:relativethresh} is a threshold result for the situation in which we select points uniformly at random on a fixed low rank variety $V(L)$, and wish to pierce the set of varieties which are of high rank \textit{relative to $V(L)$} in a sense defined shortly. We provide relevant definitions and lemmas in Subsection \ref{ss:rel-decomp}, and then state and prove Proposition \ref{prop:relativethresh} in Subsection \ref{ss:rel-thresh}.

Throughout this section, all polynomials may be assumed to be of degree at most $k$. Furthermore, when we refer to the analytic rank of a polynomial, we refer to that of its $k$-fold multilinearisation.

\subsection{Structural decomposition, relative density}\label{ss:rel-decomp}

We begin with a way to decompose an $\F_p$-vector space of polynomials into high rank and low rank components.

Let $r>0$. Given a subspace $L \leq \F_p[x_1,\ldots, x_n]_{\leq k}$, we will say that $L$ is \textit{$r$-low rank} if it possesses a basis $\cB_L$ of elements of analytic rank at most $r$, that is,
\[\max_{f \in \cB_L} \ark(f) \le r.\] Note by the triangle inequality for analytic rank (Lemma \ref{lem:triangle}), if  $L$ has dimension $\ell$ and is $r$-low rank, then 
\[\max_{f \in L} \ark(f) \le \ell r.\]
Note also that the set of $r$-low rank subspaces of $\F_p[x_1,\ldots, x_n]_{\leq k}$ is closed under summation. Therefore, in any vector subspace $S\leq \F_p[x_1,\ldots,x_n]_{\leq k}$, we have that there exists a unique maximal $r$-low rank subspace $L_{\max}$, and 
\[ L_{\max} = \sum_{r\text{-low rank } L \leq S} L.\]

In contrast, letting $S \leq \F_p[x_1,\ldots,x_n]_{\leq k}$ be a subspace, we will say that a subspace $H \leq \F_p[x_1,\ldots,x_n]_{\leq k}$ is \textit{$r$-high rank relative to $S$} if $\ark(f+g) > r $ for all nonzero $f \in H$ and all $g \in S$. Note that then $\min_{0 \neq f \in H} \ark(f) > r$ and $H \cap S = \set{0}$. Furthermore, if $H'\leq \F_p[x_1,\ldots,x_n]_{\leq k}$ satisfies $H' \cap S = \set{0}$, then $H'$ is $r$-high rank relative to $S$ if and only if  
\[\min_{f\in (H'\oplus S)\backslash S} \ark (f) > r.\]

If the $r$ to which we are referring is clear from context, we may just say ``low rank'' or ``high rank''.

\begin{lemma}[Decomposition]\label{l:decomp}
    Let $S\leq \F_p[x_1,\ldots,x_n]_{\leq k}$ be a subspace of polynomials. Then, for any $r>0$, we may decompose 
    \[S = H \oplus L,\]
    where $L$ is $r$-low rank, $H\leq \F_p[x_1,\ldots,x_n]_{\leq k}$ is $r$-high rank relative to $L$, and furthermore $L$ contains every $r$-low rank subspace in $S$.
\end{lemma}
\begin{proof}
Let $L=L_{\max}$, as defined above. Let $H$ be any subspace satisfying $S = H \oplus L$. Suppose there exists $g\in L$ and a nonzero $f \in H$ such that $\ark(f+g) \le r$. Then $\subspan_{\F_p}(f+g)$ is $r$-low rank and so is contained in $L$ by maximality. Thus  $f = (f+g) - g$ is also in $L$, a contradiction. The final statement of the lemma follows from the maximality of $L$.
\end{proof}

The next lemma gives  control on the density of varieties relative to that of their low rank part.

\begin{lemma}[Relative density]\label{lem:reldens} Let $L\leq \F_p[x_1,\ldots,x_n]_{\leq k}$ be a subspace of polynomials such that $V(L)$ is nonempty, and let $H\leq \F_p[x_1,\ldots,x_n]_{\leq k}$ be a subspace of polynomials which is $r$-high rank relative to $L$. Set $h:= \dim(H)$. Then
    $$
    \abs{\frac{\abs{V(H) \cap V(L)}}{\abs{V(L)}} - p^{-h}} \le \frac{p^{n-r/2^{k-1}}}{\abs{V(L)}}.
    $$ 
\end{lemma}
\begin{proof} Let $\ell:=\dim(L)$. Using the orthogonality of characters (see, for example, \cite[Lemma 3.2]{altfp}), we obtain
\begin{align*} 
    \abs{V(H \oplus L)} &= \sum_{v \in \FF_p^n} \EE_{f \in H \oplus L} e_p(f(v))\\
    &= p^{-(h+\ell)}\bigp{\sum_{v} \sum_{f \in L} e_p(f(v)) + \sum_{v} \sum_{f \in (H \oplus L) \setminus L} e_p(f(v))} \\
    &= p^{-h}\sum_{v} \EE_{f \in L} e_p(f(v)) + p^{-(h+\ell)} \sum_{v} \sum_{f \in (H \oplus L) \setminus L} e_p(f(v)).
\end{align*}
Then using the orthogonality of characters in the same way yields
\[\abs{\abs{V(H \oplus L)} - \abs{V(L)} p^{-h}} \leq p^{-(h+\ell)} \sum_{f\in (H\oplus L)\backslash L}\abs{\sum_v e_p(f(v))} \leq p^{n-r/2^{k-1}},\]
where the second inequality follows from Lemma \ref{lem:gauss} and the fact that $H$ is $r$-high rank relative to $L$. Dividing by $\abs{V(L)}$ gives the result.
\end{proof}

\subsection{Relative threshold}\label{ss:rel-thresh}

Using this we can state a quantitative, ``relative'' version of our main theorem. Here our random set is formed by selecting points uniformly at random \textit{on a low rank variety} $V$. On the other hand, the collection of varieties that we wish to pierce will be high rank relative to $V$. This relative version of our main theorem enjoys sufficiently good concentration around the mean for the second moment method to go through (whereas, as mentioned above, the same is not true of the absolute statement Theorem \ref{thm:mainthresh}). 

It is convenient to work with a slightly different random model: we choose the set $M$ by sampling $j$ elements of $\F_p^n$ uniformly and independently \textit{with replacement}. One may easily recover the analogous statement for the random model in which $M$ is selected uniformly at random among sets of size $|M|$.  

We note that in the following proposition, and henceforth in the paper, we may take $L=\{0\}$, and all statements go through.

\begin{proposition}[Relative threshold]\label{prop:relativethresh}
     Let $h\geq 1$, $k\geq 2$, $\ell \ge 0$, and let $r = r(n) = k2^{k+1}\log_p n$. Let $L\leq \F_p[x_1,\ldots,x_n]_{\leq k}$ be a subspace of polynomials which has dimension $\ell$, which is $r$-low rank, and for which $V(L)$ is nonempty. Let $M \subseteq V(L)$ be  chosen by sampling $j$ elements of $V(L)$ uniformly and independently with replacement. Let $\cG_{k,h, L}$ be the collection of varieties $V(H)$ generated by subspaces $H\leq \F_p[x_1,\ldots,x_n]_{\leq k}$ with $\dim(H) = h$ and such that $H$ is $r$-high rank relative to $L$. Then letting
    $$
    t_h \coloneq \frac{\log(p^h)}{\log(1+(p^h-1)^{-1})},
    $$
    we have for every $\eps > 0$  that if $j \geq (t_h+\eps){n \choose k}$ and $n$ is sufficiently large depending on $\eps$, $h$, $\ell$, $k$, $p$, then
    $$
    \pr{M \text{ pierces } \cG_{k,h,L}} \ge 1-\exp \bigp{-\frac{\eps}{2}\log(1+(p^h-1)^{-1}) {n \choose k}}.
    $$
    Furthermore, if $j \leq (t_h-\eps) {n \choose k}$, then
    $$
    \pr{M \text{ pierces } \cG_{k,h,L}} = o(1),
    $$
    as $n\to \infty$. 
\end{proposition}
\begin{proof}
    Let $\PP_L$ be the uniform measure over $V(L)$, let $j$ be a positive integer and let $M_j$ be the set formed after $j$ uniformly random selections in $V(L)$. 
    Let 
    \[X_j := \sum_{H \text{ s.t. } V(H) \in \cG_{k,h,L}} 1_{M_j \cap V(H) = \emptyset},\]
    that is, $X_j$ is the number of $H\leq \F_p[x_1,\ldots,x_n]_{\leq k}$ such that $V(H)$ is disjoint from $M=M_j$. Note that even though $X_j$ may overcount the number of varieties disjoint from $M$, we still have that $X_j = 0$ if and only if $M$ pierces $\cG_{k,h,L}$.
    
    \vspace{0.4cm}
    \textit{Piercing with high probability, quantitatively.}
    We will begin with the direction that suitably large $j$ implies a large probability of piercing. Since there are $p^{n+k \choose k} = p^{(1+o(1)){n \choose k}}$ polynomials with degree at most $k$, we have by choosing a basis for $H$ that the number of $H$ with $\dim(H) = h$ is at most $p^{(h+o(1)){n \choose k}}$. Invoking the Chevalley--Warning theorem (Lemma \ref{lem:warning}), we have that 
    \[\abs{V(L)} \geq p^{n-k \ell}.\]
    Thus, we may use Lemma \ref{lem:reldens} to bound
    \begin{equation}\label{eq:cw-lazy}
        \PP_L(v \in V(H)) \geq p^{-h} - p^{k\ell -r/2^{k-1}}\geq p^{-h} - o(1).
    \end{equation}
    We can therefore compute
    \begin{align*}
    \log \EE X_j &= \log \bigp{\sum_{H: V(H) \in \cG_{k,h,L}} (1-\PP_L(v \in V(H)))^j}\\
    &\le \log \bigp{\sum_{H} (1-p^{-h} + o(1))^j}\\
    &\le \log (p^h) (1+o(1)) {n \choose k} + j\log\bigp{1-p^{-h} + o(1)}.
\end{align*}
Noting that 
\[t_h = \frac{-\log(p^h)}{\log(1-p^{-h})},\]   
we see that if $j \geq (t_h+\eps){n \choose k}$ then 
\[\log\EE X_j \le \bigg( o(1) -\eps \log\bigp{1+(p^h-1)^{-1}}(1+o(1)) \bigg){n \choose k}.\]
The first direction then follows by taking $n$ suitably large in terms of all other parameters and invoking Markov's inequality.

\vspace{0.4cm} \textit{Not piercing with high probability.}
For the other direction, we wish to show that 
\[ \PP(X_j = 0) = o(1),\]
when $j \le (t_h - \eps){n \choose k}$. By the Paley--Zygmund inequality, it suffices to show that 
\begin{equation}\label{eq:ratio}
    \frac{\EE X_j^2}{\bigp{\EE X_j}^2} \le 1+o(1).
\end{equation}

We begin by determining a lower bound on $\EE(X_j)$. First note that the number of subspaces of polynomials of degree at most $k$ which have dimension $h$ is at least $p^{h(1-o(1))\binom{n}{k}}$. We wish to upper bound the number of those which are not $r$-high rank relative to $L$. If $H$ is not $r$-high rank relative to $L$ then either $H\cap L \ne \{0\}$, or there is $f\in (H+L)\backslash L$ with $\ark(f)\leq r$. In the latter case, say $f = f_H + f_L$, where $f_H\ne 0$. Then by the triangle inequality (Lemma \ref{lem:triangle}), $\ark(f_H) \leq  (\ell+1) r$. In the former case, it is clearly the case  that there exists $f_H\in H$ with $\ark(f_H) \leq \ell r$. Thus, choosing a basis for $H$, the number of choices for $H$ is at most the number of choices of $h$ polynomials, at least one of which has analytic rank at most $(\ell+1) r$. By Lemma \ref{lem:lowrankcount}, there are at most $p^{(h-1)(1+o(1))\binom{n}{k}}$ such choices. 
Thus, invoking Lemmas \ref{lem:reldens} and \ref{lem:gauss} in the second line below to estimate $\PP_L\bigp{v \in V(H)}$ (where we have from definitions that $H$ is $r$-high rank relative to $L$), we get
\begin{align}
\EE(X_j) &= \sum_{H: V(H)\in \cG_{k,h,L}} \bigp{1-\PP_L\bigp{v \in V(H)}}^j \nonumber\\
&\ge \bigp{p^{(h-o(1)){n \choose k}}-p^{(h-1+o(1)){n \choose k}}}\bigp{1-p^{-h} - o(1)}^j \nonumber \\
&= (1-o(1))p^{(h+o(1)){n \choose k}}(1-p^{-h} - o(1))^j. \label{eq:ex-lower}
\end{align}

To show \eqref{eq:ratio} we must of course also upper bound the quantity
\begin{equation}\label{eq:2nd-mom}
\EE(X_j^2) = \sum_{H_1,H_2: V(H_1),V(H_2) \in \cG_{k,h,L}}\PP_L(v \notin V(H_1) \cap  v \notin V(H_2))^j.
\end{equation}
We will stratify the analysis by the dimension of the high rank component in a high rank-low rank decomposition of the vector space $H_1 + H_2 + L$. Given $H_1, H_2$, invoke Lemma \ref{l:decomp} to obtain 
\begin{equation}\label{eq:iter-decomp}
    H_1 + H_2 + L = \wt H \oplus \wt L,
\end{equation}
where $\wt L$ is $r$-low rank and $\wt H$ is $r$-high rank relative to $\wt L$. Let $\wt h := \dim \wt H$ and $\wt \ell := \dim \wt L$. We have 
\begin{equation}\label{eq:dim-ineqs}
\wt \ell \geq \ell, \qquad \wt h + \wt \ell \leq 2h + \ell,
\end{equation}
where the first inequality is provided by the output of Lemma \ref{l:decomp}, and the second is immediate from \eqref{eq:iter-decomp}. Our casework will be based upon the value of $\wt h$.

\vspace{0.4cm} \textit{The main contribution.}
For appropriate $j$, the main contribution to \eqref{eq:ratio} will come from the case in which $\wt h = 2h$. If $(H_1,H_2)$ is such that the output
of Lemma \ref{l:decomp} yields $\wt h = 2h$ then we will call $(H_1,H_2)$ \textit{approximately independent}, and denote this by $H_1 \perp H_2$.  We will deal with this first, demonstrating a contribution of $1+o(1)$ to \eqref{eq:ratio}, and then later show that the remaining pairs $(H_1,H_2)$ contribute negligibly.

Note that if $\wt h = 2h$, then from the  inequalities \eqref{eq:dim-ineqs} we see that $\ell \leq \wt \ell \leq \ell$, that is, $\wt \ell = \ell$. Thus, invoking the output of Lemma \ref{l:decomp} which gives $\wt L \geq L$, we in fact have that $\wt L = L$.  Ultimately we can conclude that if $\wt h = 2h$ then 
\[ \wt H \oplus L = \wt H \oplus \wt L = H_1 + H_2 + L \cong H_1 \oplus H_2 \oplus L,\]
where the isomorphism follows from the fact that we have proven $\wt h + \wt \ell = 2h + \ell$. Therefore, if $f\in (H_1 \oplus H_2 \oplus L)\backslash L$ then $f \in  (\wt H \oplus \wt L)\backslash \wt L$, so $\ark(f) \geq r$. Then by definition, $H_1 \oplus H_2$ is $r$-high rank relative to $L$.

With this observation, we may now bound the terms corresponding to such $(H_1,H_2)$ that arise in \eqref{eq:2nd-mom}. For such pairs, invoking Lemma \ref{lem:reldens} for the first inequality and subsequently arguing similarly to \eqref{eq:cw-lazy} for the second inequality (using that $n$ is sufficiently large in terms of other parameters),
\begin{align*}
    &\PP_L(v \in V(H_1 \oplus H_2)) - \PP_L(v \in V(H_1))\PP_L(v \in V(H_2)) \\
    &\le p^{-2h} - p^{-2h} + 4\cdot \frac{p^{n-r/2^{k-1}}}{\abs{V(L)}} \leq 4p^{-2k\log_p n} = 4n^{-2k}.
\end{align*}
  Lemma \ref{lem:reldens} of course also yields the crude bound $\PP_L\bigp{v \notin V(H_1)} = 1-p^{-h} + o(1) \ge \frac 1 2$ since $p$ is an odd prime and $h \ge 1$, and the same holds for $H_2$. Thus
\begin{align*}
    &\frac{\PP_L(v \notin V(H_1) \ \cap \ v \notin V(H_2))}{\PP_L(v \notin V(H_1))\PP_L(v \notin V(H_2))} \\
    &= 1  + \frac{\PP_L(v \in V(H_1 \oplus H_2)) - \PP_L(v \in V(H_1))\PP_L(v \in V(H_2))}{\PP_L(v \notin V(H_1))\PP_L(v \notin V(H_2))} \\
    &\le 1+16n^{-2k}.
\end{align*}
So then
\begin{align*}
    &\frac{\sum_{H_1\perp H_2}\PP_L(v \notin V(H_1)\ \cap \ v \notin V(H_2))^j}{(\EE X_j)^2}\\
    &= \frac{\sum_{H_1\perp H_2} (\PP_L(v \notin V(H_1))\PP_L(v \notin V(H_2)))^j \left( \frac{\PP_L(v \notin V(H_1) \ \cap \ v \notin V(H_2))}{\PP_L(v \notin V(H_1))\PP_L(v \notin V(H_2))}\right)^j}{\sum_{H_1,H_2}(\PP_L(v \notin V(H_1))\PP_L(v \notin V(H_2)))^j} \\
    &\leq (1 + 16n^{-2k})^j.
\end{align*}
When $j=O(n^k)$ we have that this is at most $1+o(1)$, completing the argument for the situation in which the pair $(H_1, H_2)$ is approximately independent.

\vspace{0.4cm} \textit{The error terms.} In the remaining cases we have that $H_1$ and $H_2$ are not approximately independent. We will show the contribution of all such pairs to \eqref{eq:ratio} is $o(1)$ for appropriate $j$. Thus from \eqref{eq:dim-ineqs} we henceforth have $0\leq  \wt h \leq  2h-1$; fix such $\wt h$ and let $\cH_{\wt h}$ be the collection of pairs $(H_1,H_2)$ which yield this value of $\wt h$ under Lemma~\ref{l:decomp}.

We claim in this paragraph that $|\cH_{\wt h}|$ is at most $p^{\wt h(1+o(1))\binom{n}{k}}$. Firstly, the number of polynomials of degree at most $k$ in $\F_p[x_1,\ldots, x_n]$ is  $p^{\binom{n}{k}(1+o(1))}$. Next, by choosing a basis for $\wt H$, there are at most $p^{\wt h(1+o(1)) \binom{n}{k}}$ choices for $\wt H$. Furthermore, since $\wt L$ is $r$-low rank, it by definition possesses a basis of polynomials of analytic rank at most $r\ll_{p,k} \log n$. By Lemma~\ref{lem:lowrankcount}, there are $p^{O_{p,k}(n^{k-1}\log^2 n)}$ such polynomials. Thus the number of choices for $\wt H \oplus \wt L$ is at most $p^{\wt h(1+o(1)) \binom{n}{k}}$. Given $\wt H \oplus \wt L$, the number of choices for $(H_1,H_2)$ such that $H_1 + H_2 + L = \wt H \oplus \wt L$ is of course $\ll_{p,h,\ell} 1$.

Next we obtain a uniform bound on the contribution to \eqref{eq:2nd-mom} of those $(H_1,H_2)\in \cH_{\wt h}$. We have that 
\[\PP_L(v \notin V(H_1) \cap v \notin V(H_2)) = 1- \PP_L(v \in V(H_1)) - \PP_L(v \in V(H_2)) + \PP_L(v \in V(H_1 + H_2));\]
we now pursue an upper bound on $\PP_L(v \in V(H_1 + H_2))$. To this end note that if $\PP_L(v \in V(H_1 + H_2)) \ne 0$ then $V(\wt L)\ne \emptyset$ and so
\begin{align*}
\PP_L(v \in V(H_1 + H_2)) &= \frac{|V(H_1 + H_2 + L)|}{|V(L)|}\\
&= \frac{|V(\wt H \oplus \wt L)|}{|V(L)|} \\
&\leq \frac{|V(\wt H \oplus \wt L)|}{|V(\wt L)|},
\end{align*}
where the inequality holds since Lemma~\ref{l:decomp} produces $\wt L \geq L$. We are now in a position to invoke our relative density lemma, Lemma \ref{lem:reldens}, which yields, again under the assumption that $\PP_L(v \in V(H_1 + H_2)) \ne 0$,
\begin{equation}\label{eq:h1h2-prob}
\PP_L(v \in V(H_1 + H_2)) \leq p^{-\wt h} + p^{n-r/2^{k-1}}|V(\wt L)|^{-1} \leq p^{-\wt h} + p^{k(2h+\ell)-r/2^{k-1}} = p^{-\wt h} + o(1),
\end{equation}
where the second inequality is the Chevalley--Warning theorem (Lemma \ref{lem:warning}). Of course the bound in \eqref{eq:h1h2-prob} also holds in the situation that $\PP_L(v \in V(H_1 + H_2)) = 0$. For $(H_1,H_2)\in \cH_{\wt h}$, using Lemmas \ref{lem:warning} and \ref{lem:reldens} again in the same way we may conclude
\begin{align}\label{eq:prelim-bound}
    \PP_L(v \notin V(H_1) \cap v \notin V(H_2)) &= 1- \PP_L(v \in V(H_1)) - \PP_L(v \in V(H_2)) + \PP_L(v \in V(H_1 + H_2)) \nonumber \\
    &\le 1-2p^{-h} + p^{-\wt h} + o(1).
\end{align}
We also have the trivial bound 
\[ \PP_L(v \notin V(H_1) \cap v \notin V(H_2))\leq 1- \PP_L(v\in V(H_1)),\]
which, invoking Lemmas \ref{lem:warning} and \ref{lem:reldens} once again, we see is of size $1-p^{-h} + o(1)$. Writing this as $1-2p^{-h} + p^{-h} + o(1)$ and combining with \eqref{eq:prelim-bound} we have 
\[ \PP_L(v \notin V(H_1) \cap v \notin V(H_2)) \leq 1- 2p^{-h} + p^{-\max(h,\wt h)} + o(1).\]

The previous two paragraphs together with the lower bound on $\EE(X_j)$ in \eqref{eq:ex-lower} show that the contribution of $\cH_{\wt h}$ to \eqref{eq:ratio} is bounded above by
\[ \ll \frac{p^{(\wt h+o(1))\binom{n}{k}}\left(1-2p^{-h} + p^{-\max(h,\wt h)} + o(1)\right)^j}{p^{(2h-o(1))\binom{n}{k}}(1-p^{-h} - o(1))^{2j}}.\]
By summing this contribution over $0\leq \wt h \leq  2h-1 $, it suffices to show that the value of the above expression is of size $o(1)$ for all $\wt h$ in this range. To this end, the logarithm of the above expression is
\[ (\wt h-2h +o(1))(\log p) {n \choose k} + j\log \left( 1-2p^{-h}+p^{-\max(h,\wt h)}\right)(1+o(1)) - 2j\log(1-p^{-h})(1+o(1)).\]
Setting \[j\leq (1-\delta)t_h\binom{n}{k} = (1-\delta)\frac{\log (p^h)}{-\log(1-p^{-h})}\binom{n}{k}\]
yields the following upper bound for the previous expression:
\[\left( \wt h-2\delta h- (1-\delta)h \cdot \frac{\log(1- 2p^{-h} + p^{-\max(h,\wt h)})}{\log(1-p^{-h})}+o(1)\right)(\log p) {n \choose k}.\]
Dividing by $h\cdot (\log p)\binom{n}{k}$ and setting $\beta = \wt h/h$, $x=p^{-h}$, we are left with the task of verifying that the following expression is negative for all $\beta \in [0,2-1/h]$, $x\in(0,1/3]$ and all $\delta \in (0,1)$:
\begin{equation}\label{eq:Phi}\beta - 2\delta - (1-\delta)\frac{\log(1- 2x + x^{\max(1,\beta)})}{\log(1-x)} =: \Phi(x,\beta,\delta).\end{equation}
Clearly $\Phi(x,\beta,\delta)$ is increasing in $\beta$ when $\beta \in [0,1]$, so it suffices to prove that $\Phi$ is negative in the restricted domain $\beta\in[1,2-1/h]$. Now, for all $x,\delta$ in the relevant domain and $\beta \in [1,2-1/h]$, differentiating twice gives that \[\frac{d^2}{d\beta^2} \Phi(x,\beta,\delta) = -(1-\delta) \frac{x^\beta (1-2x)(\log x)^2}{\log(1-x)\cdot (1-2x+x^\beta)^2}>0.\] Also, 
\[\Phi(x,1,\delta) = -\delta, \qquad \Phi(x,2,\delta) = 0,\]
for all appropriate $x,\delta$. This allows us to conclude that $\Phi(x,\beta,\delta)<0$ on the relevant domain, completing the proof that the pairs $(H_1, H_2)$ which are not approximately independent contribute $o(1)$ to \eqref{eq:ratio}, and thus completing the proof of the proposition. 
\end{proof}

\begin{remark} \label{rem:homog}
We note that the previous proof for failing to pierce $\cG_{k,h,L}$ with high probability yields the slightly stronger statement that sampling the same number of points asymptotically almost surely fails to intersect a variety generated by a subspace of dimension $h$ consisting of \textit{homogeneous} polynomials of degree \textit{exactly} $k$ that are high rank relative to $L$. We will use this stronger statement for the application to the random Szemer\'edi problem in Section \ref{sec:rszproof}.  
\end{remark}

\section{Proof of Theorem \ref{thm:mainthresh}}\label{s:main-proof}

\subsection{Thresholds for sub-collections}

Throughout this subsection we will take \\$r = k2^{k+1} \log_p n$, and $S$, $H$, and $L$ will denote subspaces of $\F_p[x_1,\ldots,x_n]_{\leq k}$. We will let $M$ be the union of $j$ elements taken uniformly and independently with replacement from $\FF_p^n$. Note that this $M$ is different from the one in the statement of Theorem \ref{thm:mainthresh}; we will transfer results to the fixed size model later in the proof of Theorem \ref{thm:mainthresh}, which is contained in the next subsection.

In the case of hypersurfaces, i.e. $s=1$, the generating polynomial is either high rank in which case Proposition \ref{prop:relativethresh} applies with $L=\{0\}$, or low rank, in which case there are sufficiently few such hypersurfaces. For $s > 1$ this dichotomy no longer holds. In the previous section we proved a threshold result for the piercing of varieties which are high rank relative to a fixed low rank variety, within which we sampled. We now transfer this to a threshold result for sampling within all of $\F_p^n$ and piercing all varieties with a fixed low rank component.

\begin{lemma}[Threshold for fixed low-rank component]\label{lem:chernoffthresh}
    Fix an $r$-low rank subspace $L$ consisting of polynomials of degree at most $k$ and such that $V(L)$ is nonempty. Let $\alpha := |V(L)|\cdot p^{-n}$. Let $\cH_{k,h, L}$ be the collection of varieties generated by subspaces $S = H \oplus L$ where $H$ consists of polynomials of degree at most $k$, is $r$-high rank relative to $L$, and has dimension $h \geq 1$. For every $\eps \in (0,1)$, if $j \geq (\alpha^{-1}t_h+\eps){n \choose k}$, then for all $n$ sufficiently large in terms of all other parameters,
    $$
    \pr{M \text{ pierces } \cH_{k,h,L}} \ge 1-2\exp \bigp{-\frac{(\alpha \eps)^2}{8(t_h+\alpha \eps)}{n \choose k}},
    $$
    and if $j \leq (\alpha^{-1}t_h-\eps) {n \choose k}$ then
    $$
    \pr{M \text{ pierces } \cH_{k,h,L}} = o(1)
    $$
    as $n\to \infty$.
\end{lemma}
\begin{proof}
    Since each variety defined by an element of $\cH_{k,h,L}$ is contained in $V(L)$, we have that $M$ pierces $\cH_{k,h,L}$ if and only if $M \cap V(L)$ pierces $\cG_{k,h,L}$, as defined in Proposition \ref{prop:relativethresh}. Letting $\delta \coloneqq \alpha \eps / 2$, consider sampling at least $\alpha^{-1}(t_h + 2\delta){n \choose k}$ points, i.e., set $j \ge \alpha^{-1}(t_h + 2\delta){n \choose k}$. Recall Lemma~\ref{lem:warning} which says that $\alpha \ge p^{-\ell k} \gg_{p,k,\ell} 1$, and so $\delta \gg_{p,k,\ell,\eps} 1$. Let $X_L$ be the number of points (not necessarily distinct) of $M$ landing in $V(L)$. Then by a Chernoff bound \cite[Theorem A.1.13]{probmethod}
$$
\pr{X_L - (t_h + 2\delta){n \choose k} < -\delta{n \choose k}} \le \exp \bigp{-\frac{\delta^2}{2(t_h+2\delta)}{n \choose k}} = \exp \bigp{-\frac{(\alpha \eps)^2}{8(t_h+\alpha \eps)}{n \choose k}}.
$$
Then assuming $X_L \ge (t_h + \delta){n \choose k}$, note that each of the $X_L$ points we are sampling to make $M \cap V(L)$ are distributed uniformly and independently inside $V(L)$ with replacement. Thus we have by Proposition \ref{prop:relativethresh} (taking $j = X_L$ and recalling that $\delta \gg 1$) that the probability there exists $H$ of dimension $h$ which is $r$-high rank relative to $L$ and such that $V(H) \cap V(L) \cap M = \emptyset$ is at most
$$
    \exp \bigp{-\frac{\delta}{2}\log(1+(p^h-1)^{-1}) {n \choose k}} = \exp \bigp{-\frac{\alpha \eps \log(p^h)}{4 t_h} {n \choose k}}.
$$
Summing these two probabilities gives the first part of the result.

For the other side of the threshold, sample $\alpha^{-1}(t_h - 2\delta){n \choose k}$ points. Then again by a Chernoff bound \cite[Theorem A.1.4]{probmethod},
$$
\pr{X_L - (t_h - 2\delta){n \choose k} > \delta{n \choose k}} < \exp \bigp{-\frac{2\delta^2}{\alpha^{-1}(t_h-2\delta)}{n \choose k}} = o(1).
$$
Then assuming $X_L \le (t_h - \delta){n \choose k}$ we have by Proposition \ref{prop:relativethresh} that with high probability there exists an $H$ with $\dim(H) = h$, which is $r$-high rank relative to $L$ and which satisfies $V(H) \cap V(L) \cap M = \emptyset$. Setting $S \coloneqq H+L \cong H\oplus L$ (since $H$ is $r$-high rank relative to $L$)  we have $V(S) \cap M = \emptyset$ with high probability.
\end{proof}

Finally, applying the above and taking a union over low rank subspaces gives a threshold for all $S = H \oplus L$ with fixed $\dim (H), \dim (L)$.

\begin{lemma}[Threshold for fixed rank parameters]\label{lem:lowdimthresh}
     Let $\cI_{k,h,\ell}$ be the collection of nonempty varieties generated by subspaces $S=H \oplus L$ with $\dim(H) = h \geq 1$, $\dim(L) = \ell$,  consisting of polynomials of degree at most $k$ and such that $H$ is $r$-high rank relative to $L$, and $L$ is $r$-low rank. Then letting
     $$
     d_{h,\ell} \coloneqq p^{\ell k} \cdot t_h = p^{\ell k} \cdot \frac{\log (p^h)}{\log(1+(p^h-1)^{-1})},
     $$
we have for every $\eps >0$
\[ \pr{M \text{ pierces } \cI_{k,h,\ell}} = \begin{cases} 
      1-o(1), & j \geq (d_{h,\ell}+\eps){n \choose k} \\
      o(1), & j \leq (d_{h,\ell}-\eps) {n \choose k} 
   \end{cases},\]
   as $n \to \infty$.
\end{lemma}
\begin{proof}
    To show the statement about failing to pierce $\cI_{k,h,\ell}$, we use Lemma \ref{lem:warn-tight}, which says that the Chevalley--Warning theorem is tight. This yields $L$ with $\dim(L) = \ell$ such that $\abs{V(L)} = p^{-k \ell}p^n$; we note that this $L$ is $r$-low rank. Thus if
    $$
    j \leq (d_{h,\ell}-\eps) {n \choose k} = (p^{k \ell}t_h-\eps) {n \choose k},
    $$
    then by Lemma \ref{lem:chernoffthresh} we have that $\cH_{k,h,L}$ is not pierced with high probability. But  $\cH_{k,h,L} \subseteq \cI_{k,h,\ell}$, so $\cI_{k,h,\ell}$ is also not pierced.

    For the other direction we have
    $$
    \cI_{k,h,\ell} = \bigcup_{\substack{L: \dim(L) = \ell \\ L \text{ is $r$-low rank}, V(L)\ne \emptyset}} \cH_{k,h,L},
    $$
    so to pierce $\cI_{k,h,\ell}$ it suffices to pierce each such $\cH_{k,h,L}$. Consider any such $L$ in the union above. Then $\abs{V(L)} = \alpha_L p^n \ge p^{-\ell k}p^n$ by Lemma \ref{lem:warning}. Thus if
    $$
    j \geq (d_{h,\ell}+\eps) {n \choose k} \ge (\alpha_L^{-1}t_h+\eps) {n \choose k},
    $$
    then by Lemma \ref{lem:chernoffthresh} we have that the probability that $\cH_{k,h,L}$ is not pierced is at most $2\exp \bigp{-\frac{(\alpha_L \eps)^2}{8(t_h+\alpha_L \eps)}{n \choose k}}$. Taking the union bound over the $p^{o(n^k)}$ $r$-low rank $L$ of dimension $\ell$ (cf. Lemma \ref{lem:lowrankcount}) gives that the probability that $\cI_{k,h,\ell}$ is not pierced is at most \[p^{o(n^k)} \cdot 2\exp \bigp{-\frac{(\alpha_L \eps)^2}{8(t_h+\alpha_L \eps)}{n \choose k}} = o(1).\]
    This completes the proof.
\end{proof}

\subsection{Proof of Theorem \ref{thm:mainthresh}}\label{ss:proof-thm}
Now we complete the proof of Theorem \ref{thm:mainthresh}. Here we will take $M$ to be the fixed size random subset in the statement of the theorem, and $M'$ the union of $j$ elements taken uniformly and independently with replacement. Thus the previously established lemmas and proposition apply to $M'$; we will first show the result for $M'$ and later transfer to $M$. 

We will need an analytic lemma which will determine the bottleneck for the threshold result of Theorem \ref{thm:mainthresh}. Recall in the following that 
\[ d_{h,\ell} := p^{\ell k} \cdot \frac{\log (p^h)}{\log(1+(p^h-1)^{-1})} = p^{\ell k} \cdot \frac{\log (p^h)}{-\log(1-p^{-h})}.\]

\begin{lemma}\label{lem:dhl-sizes}
        Let $s, h \geq 1$, $\ell \geq 0$, $k\geq 2$ be integers such that $s=h+\ell$. Then, for fixed $s$, $d_{h,\ell}$ is decreasing in $h$. That is, for all $1\leq h \leq s-1$,
        \[d_{h+1,s-(h+1)} \leq d_{h, s-h}.\]
\end{lemma}
\begin{proof}
    Observing the bounds 
    \[ x \leq -\log(1-x) \leq \frac{x}{1-x}, \qquad x\in [0,1),\]
    we can bound the ratio
    \[\frac{d_{h+1,s-(h+1)}}{d_{h, s-h}} = \frac{h+1}{hp^k} \cdot \frac{-\log(1-p^{-h})}{-\log(1-p^{-(h+1)})} \leq \frac{h+1}{hp^k} \cdot \frac{p}{1-p^{-h}} \leq \frac{h+1}{h(p-1)} \leq 1,\]
    where we have used that $k\geq 2$, $p\geq 3$, $h\geq 1$.
\end{proof}

\begin{proof}[Proof of Theorem \ref{thm:mainthresh}] Let $\{f_1,\dots,f_s\}\subset \F_p[x_1,\ldots, x_n]_{\leq k}$ be a set of $s$ polynomials (we may insert zero polynomials to ensure the set is of size $s$ without changing the corresponding variety) such that $V(f_1,\ldots, f_s)$ is nonempty, and let $S = \subspan(f_1,\dots,f_s)$. Let $\wt s:= \dim(S) \le s$ and $V(S) = V(f_1,\dots,f_s)$. In this way it is equivalent to consider the threshold for piercing all subspaces of dimension at most $s$, consisting of polynomials of degree at most $k$. In a slight abuse of notation we will redefine $\cF_{k,s}$ to denote this collection, noting that it possesses the same piercing threshold as the original definition. For each such $S$ we  decompose, via Lemma \ref{l:decomp}, $S = H\oplus L$ where $L$ is $r$-low rank (with $r = k2^{k+1}\log_p n$) and $H$ is $r$-high rank relative to $L$. Let $\ell = \dim(L), h = \dim(H)$ and $t_h, d_{h,\ell}$ be as defined in Proposition \ref{prop:relativethresh} and Lemma \ref{lem:lowdimthresh} respectively. Recall from the statement of Theorem \ref{thm:mainthresh} the notation $c_{k,s} = p^{(s-1)k}t_1 = d_{1,s-1}$. Thus we aim to show that for any $\eps > 0$
$$
\pr{M' \text{ pierces } \cF_{k,s}} =  \begin{cases} 
      1-o(1), & j \geq (d_{1,s-1}+\eps){n \choose k} \\
      o(1), & j \leq (d_{1,s-1}-\eps) {n \choose k} 
   \end{cases}.
$$
The latter of the two statements is immediate from Lemma \ref{lem:lowdimthresh}: for $j \le (d_{1,s-1} - \eps){n \choose k}$ we have that $\cI_{k,1,s-1}$ is not pierced with high probability, so $\cF_{k,s} \supseteq \cI_{k,1,s-1}$ is also not pierced.

For the former statement we have 
$$\cF_{k,s} = \bigcup_{\wt s = 0}^{s} \bigcup_{\ell = 0}^{\wt s} \cI_{k,\wt s-\ell,\ell},
$$
    so to pierce $\cF_{k,s}$ it suffices to pierce each such $\cI_{k,\wt s -\ell,\ell}$. The case $\ell = \wt s$ is not covered by our previous lemmas so we now argue this case directly. 
    By our bound on the number of 
    $r$-low rank polynomials Lemma \ref{lem:lowrankcount} and by choosing a basis for $S$, we observe that 
    \[\left | \cI_{k,0,\wt s} \right| = p^{o(n^k)}.\] Thus by the Chevalley--Warning theorem (Lemma \ref{lem:warning}) we have
\begin{align*}
    \sum_{S: \dim(L) = \dim(S)} (1-\PP(v \in V(S)))^j \le
    p^{o(n^k)}(1-p^{-sk})^j = o(1),
\end{align*}
since $j \gg n^k$. That is, the expected number of $S$ with $\ell = \wt s$ and with $V(S) \cap M' = \emptyset$ is $o(1)$, and so we may infer the corresponding probability statement in this case by Markov's inequality. If $\ell < \wt s$, then by the analysis in Lemma \ref{lem:dhl-sizes},
    \begin{equation}\label{eq:maxthresh}
        j \geq (d_{1,s-1}+\eps) {n \choose k} \ge (d_{s-\ell,\ell}+\eps) {n \choose k} \ge (d_{\wt s-\ell, \ell}+\eps) {n \choose k},
    \end{equation}
    and so by the fixed rank parameter threshold result (Lemma \ref{lem:lowdimthresh}) we have that the probability $\cI_{k,\wt s -\ell, \ell}$ is not pierced is $o(1)$. Taking the union bound over the $<s^2$ choices of $\wt s,\ell$ gives that the probability $\cF_{k,s}$ is not pierced is $ o(1)$, completing the proof of the former statement.

    Finally we deduce the result for a random $M$ of fixed size from the result for $M'$ sampled via choosing $j$ elements uniformly with replacement. This follows simply by noting that when $j = O(n^k)$, we have that the probability of a replaced element being reselected in $M'$ is $o(1)$. Thus the corresponding statements for $M$ holds in the regime $|M|=O(n^k)$, and then one may of course use that the property of piercing is monotonic in $|M|$. 
    \end{proof}

\section{Proof of random Szemer\'{e}di bound}\label{sec:rszproof}

In this section, we prove Theorem \ref{thm:rszbound}, as well as subsequently providing an observation about its optimality in Proposition \ref{prop:densityupper}.
In the proof of Theorem \ref{thm:rszbound} we will need the following bound which comes out of the above analysis.

\begin{proposition}[Random Szemer\'{e}di construction]\label{p:rszvariety}
    Let $p > k \ge 2, t\geq 1$ be fixed. Let $\eps > 0$. If $M \subseteq \FF_p^n$ is chosen uniformly at random of size
    $$
    \abs{M} \le \bigp{\frac{\log (p^t)}{\log(1+(p^t-1)^{-1})}-\eps}{n \choose k},
    $$
    then with high probability there is a variety $A \subseteq \FF_p^n$ determined by the vanishing of homogeneous polynomials of degree $k$ such that $A \cap M = \emptyset$, and with $\abs{A} \ge (p^{-t}-o(1))  p^n$.
\end{proposition}
\begin{proof}
    Let $r = k2^{k+1}\log_p n$, as we have had in the previous section. Let $\cF_t$ be the collection of varieties $V(H)$ generated by subspaces $H$ with $\dim(H) = t$ consisting of homogeneous polynomials of degree $k$ and such that $H$ is $r$-high rank relative to $L=\{0\}$, that is, $\min_{0 \neq f \in H} \ark(f) > r$. Invoking Lemma \ref{lem:reldens} with $L = \set{0}$ and $h=t$, each such variety then has density at least $p^{-t} - o(1)$. Proposition \ref{prop:relativethresh} with $L = \set{0}$ says that if $M'$ is formed by sampling
    \[\bigp{\frac{\log (p^t)}{\log(1+(p^t-1)^{-1})}-\eps/2}{n \choose k}\]
    elements of $\F_p^n$ uniformly and independently with replacement, then asymptotically almost surely $\cF_t$ is not pierced by $M'$ (recall that the threshold is the same when requiring homogeneity; cf. Remark \ref{rem:homog}). This implies the proposition by an argument much the same as in Theorem \ref{thm:mainthresh}. 
\end{proof}

We are now  ready to recall and prove the following. Note that the meaning of $k$ changes in the following statement: it now refers to the length of the arithmetic progression, which is one more than the degree of the corresponding varieties.

\rsz* 
\begin{proof}
Denoting the $k$-fold derivative $\Delta_d\cdots \Delta_d$ by $\Delta_d^k$, an induction on $k$ reveals that if $f: \F_p^n \to \F_p$ then, 
\[\Delta_d^k f(x) = \sum_{j=0}^k (-1)^{k-j} {k \choose j} f(x+jd).\]
If furthermore 
$f \in \F_p[x_1,\ldots,x_n]$ has degree $k$, then $\Delta_d^k f(x)$ is constant and so 
\[\sum_{j=0}^k (-1)^{k-j} {k \choose j} f(x+jd) = \Delta_d^k f(x) = \Delta_d^k f(0) = \sum_{j=0}^k (-1)^{k-j} {k \choose j} f(jd).\]
If $f$ is \textit{homogeneous}, then this gives
\[ \sum_{j=0}^k (-1)^{k-j} {k \choose j} f(x+jd) = \left(\sum_{j=0}^k (-1)^{k-j} {k \choose j} j^k \right)f(d) = k! f(d).\]
Therefore, for a homogeneous degree $k$ polynomial $f$, if $x,x+d,\dots,x+kd \in V(f)$ then also $d \in V(f)$ (using that $k! \not \equiv 0 \pmod p$). Iterating this for multiple polynomials yields the observation that, setting $A:= V(f_1,\ldots, f_s)$ for homogeneous degree $k$ polynomials $f_i$, if $A$ contains a $(k+1)$-AP with common difference $d$, then $d\in A$. 

Thus, to find a subset of $\FF_p^n$ with no $(k+1)$-AP with common difference in $M$ it suffices to find a variety generated by homogeneous degree $k$ polynomials which is disjoint from $M$. This occurs with high probability if $\abs{M}$ is below the piercing threshold for the collection of $\alpha$-dense varieties which are generated by homogeneous degree $k$ polynomials. Therefore, the theorem now follows from Proposition \ref{p:rszvariety} (and replacing $k$ by $k-1$). 
\end{proof}

We conclude by noting in the upcoming proposition that if $\abs{M} = \omega(n^k)$ random differences are sampled then with high probability every dense variety has nonempty intersection with $M$. For brevity we sample $M$ with replacement, but one may of course recover the analogous statement for random $M$ of fixed size. 

\begin{proposition}\label{prop:densityupper}
    Let $p > k \ge 3$ and $\alpha \in (0,1)$ be fixed. If $M \subseteq \FF_p^n$ is chosen by sampling $|M|$ elements of $\F_p^n$ at random with replacement, then with  probability $1-o_{\frac{|M|}{n^k}\to \infty}(1)$, it holds that $A \cap M \neq \emptyset$ for every variety $A \subseteq \FF_p^n$ determined by the vanishing of polynomials of degree at most k such that $\abs{A} \ge \alpha p^n$.
\end{proposition}
\begin{proof}
    Let $m \coloneqq \abs{M}$ and let  $r = (m/n^k)^{1/2}$. In this proof, asymptotic notation should be understood to be with respect to the limit $\frac{m}{n^k} \to \infty$ (note that since $p$ is fixed, this implies that $n\to \infty$). Implicit constants are allowed to depend on $p,k,\alpha$ without denotation.
    
    First we argue in this paragraph that the number of varieties $A$ as described in the proposition statement is at most $p^{o(m)}$. For any such $A$ consider $S \le \FF_p[x_1,\dots,x_n]_{\le k}$ such that $V(S) = A$. By \cite{baranczuk}*{Theorem 1} we may assume $\dim(S) \le n$.    Using Lemma \ref{l:decomp}, decompose $S = H \oplus L$ and let $h \coloneqq \dim(H)$. To bound the number of $V(S)$ we will bound the number of choices for $H$ and $L$. Since $H$ is $r$-high rank relative to $\set{0}$, by Lemma \ref{lem:reldens}
    $$
    p^{-h} \ge \frac{\abs{V(H)}}{p^n} - o(1) \ge \frac{\abs{V(S)}}{p^n} - o(1) \ge \alpha - o(1).
    $$
    In particular $h = O(1)$, so the number of choices for bases of $H$, and thus the number of choices for $H$,  is $p^{O(n^k)}$. Turning now to $L$, since $\dim(L) \le \dim(S) \le n$ and $L$ is $r$-low rank we may find a basis of $L$ consisting of at most $n$ polynomials of analytic rank at most $r$. Thus by Lemma \ref{lem:lowrankcount} there are $p^{O(n^k r \log_p r)}$ many choices of this basis, and hence at most that many choices for $L$. Putting the two counts together and recalling our choice of $r$, there are at most $
    p^{O(n^k r \log_p r)} \cdot p^{O(n^k)} = p^{o(m)}$
    choices for the variety $A$.
    
    The result then follows from Markov's inequality. Indeed, the expected number of $A$ such that $A \cap M = \emptyset$ is at most $p^{o(m)} (1-\alpha)^m \to 0$, and thus with probability $1-o_{\frac{m}{n^k} \to \infty}(1)$ there is no such $A$.
\end{proof}

\bibliographystyle{amsplain}
\bibliography{rsz.bib}

\end{document}